\documentclass[12pt,reqno]{amsart}

\usepackage{pdfsync}

\usepackage{amssymb}

\usepackage{mathrsfs}

\DeclareMathAlphabet{\mathpzc}{OT1}{pzc}{m}{it}

\usepackage[all]{xy}

\entrymodifiers={+!!<0pt,\fontdimen22\textfont2>}

\usepackage{tikz}
\usetikzlibrary{arrows,decorations.pathmorphing,decorations.pathreplacing,positioning,shapes.geometric,shapes.misc,decorations.markings,decorations.fractals,calc,patterns}

\usepackage{graphicx}

\def\BZ{\mathbb{Z}}

\def\fR{\mathfrak{R}}
\def\fS{\mathfrak{S}}
\def\fT{\mathfrak{T}}
\def\fU{\mathfrak{U}}

\def\fr{\mathfrak{r}}

\def\ft{\mathfrak{t}}

\def\fx{\mathfrak{x}}
\def\fy{\mathfrak{y}}

\def\sC{\mathsf{C}}

\def\sS{\mathsf{S}}
\def\sT{\mathsf{T}}

\def\add{\operatorname{add}}

\def\adots{\mathinner{\mkern1mu\raise1.0pt\vbox{\kern7.0pt\hbox{.}}\mkern2mu\raise4.0pt\hbox{.}\mkern2mu\raise7.0pt\hbox{.}\mkern1mu}}

\def\dim{\operatorname{dim}}

\def\Hom{\operatorname{Hom}}

\def\ind{\operatorname{ind}}

\newtheorem{Lemma}{Lemma}[section]

\newtheorem{Proposition}[Lemma]{Proposition}

\theoremstyle{definition}
\newtheorem{Definition}[Lemma]{Definition}

\newtheorem{Construction}[Lemma]{Construction}
\newtheorem{Remark}[Lemma]{Remark}

\newtheorem{Notation}[Lemma]{Notation}

\begin{document}

\setlength{\parindent}{0pt}
\setlength{\parskip}{7pt}

\title[Cluster tilting vs.\ weak cluster tilting]{Cluster tilting vs.\
  weak cluster tilting in Dynkin type A infinity}

\author{Thorsten Holm}
\address{Institut f\"{u}r Algebra, Zahlentheorie und Diskrete
Mathematik, Fa\-kul\-t\"at f\"ur Mathematik und Physik, Leibniz
Universit\"{a}t Hannover, Welfengarten 1, 30167 Hannover, Germany}
\email{holm@math.uni-hannover.de}
\urladdr{http://www.iazd.uni-hannover.de/\~{ }tholm}

\author{Peter J\o rgensen}
\address{School of Mathematics and Statistics,
Newcastle University, Newcastle upon Tyne NE1 7RU, United Kingdom}
\email{peter.jorgensen@ncl.ac.uk}
\urladdr{http://www.staff.ncl.ac.uk/peter.jorgensen}




\begin{abstract} 

  This paper shows a new phenomenon in higher cluster tilting theory.
  For each positive integer $d$, we exhibit a triangulated category
  $\sC$ with the following properties.

  On one hand, the $d$-cluster tilting subcategories of $\sC$ have
  very simple mutation behaviour: Each indecomposable object has
  exactly $d$ mutations.  On the other hand, the weakly $d$-cluster
  tilting subcategories of $\sC$ which lack functorial finiteness can
  have much more complicated mutation behaviour: For each $0 \leq \ell
  \leq d-1$, we show a weakly $d$-cluster tilting subcategory
  $\sT_{\ell}$ which has an indecomposable object with precisely
  $\ell$ mutations.

  The category $\sC$ is the algebraic triangulated category generated
  by a $( d+1 )$-spherical object and can be thought of as a higher
  cluster category of Dynkin type $A_{\infty}$.

\end{abstract}

\maketitle

\setcounter{section}{-1}
\section{Introduction}
\label{sec:introduction}

This paper shows a new phenomenon in higher cluster tilting theory.
For each integer $d \geq 1$, we exhibit a triangulated category $\sC$
whose $d$-cluster tilting subcategories have {\em very simple}
mutation behaviour, but whose weakly $d$-cluster tilting subcategories
can have {\em much more complicated} mutation behaviour which we can
control precisely.

To make sense of this, recall that if $\sT$ is a full subcategory of a
triangulated category, then $\sT$ is called {\em weakly $d$-cluster
tilting} if it satisfies the following conditions where $\Sigma$ is
the suspension functor.
\begin{align*}
  t \in \sT & \; \Leftrightarrow \;
    \Hom( \sT , \Sigma t ) = \cdots = \Hom( \sT , \Sigma^d t ) = 0, \\
  t \in \sT & \; \Leftrightarrow \;
    \Hom( t , \Sigma \sT ) = \cdots = \Hom( t , \Sigma^d \sT ) = 0.
\end{align*}
If $\sT$ is also left- and right-approximating in the ambient category
in the sense of Remark \ref{rmk:approximating}, then it is called {\em
  $d$-cluster tilting}.  These definitions are due to Iyama \cite{I}
and have given rise to an extensive homological theory, see for
instance \cite{BMR} and \cite{IY}.  Note that if $\sT = \add t$ for an
object $t$, then $\sT$ is automatically left- and right-approximating,
but we will study subcategories which are not of this form since they
have infinitely many isomorphism classes of indecomposable objects.

One remarkable property of $d$-cluster tilting theory is {\em
  mutation}.  If $t \in \sT$ is an indecomposable object, then it is
sometimes possible to remove $t$ from $\sT$ and insert an
indecomposable object $t^* \not\cong t$ in such a way that the
subcategory remains (weakly) $d$-cluster tilting.  This is called {\em
  mutation of $\sT$ at $t$}, see \cite[sec.\ 5]{IY}.

In good cases, there are exactly $d$ different choices of $t^*$ up to
isomorphism.  That is, there are $d$ ways of mutating $\sT$ at $t$,
see \cite[sec.\ 5]{IY}.

To be more precise, one hopes(!) that this happens for $d$-cluster
tilting subcategories.  Indeed, it does happen for $d = 1$ by
\cite[thm.\ 5.3]{IY}, but can fail for $d \geq 2$, see \cite[thms.\
9.3 and 10.2]{IY}.  The situation for weakly $d$-cluster tilting
subcategories is less clear.

We can now explain the opening paragraph of the paper.  Let us first
define $\sC$ which, as we will explain below, can be thought of as a
$d$-cluster category of type $A_{ \infty }$. 

\begin{Definition}
\label{def:blanket}
For the rest of the paper, $k$ is an algebraically closed field,
$d \geq 1$ is an integer, and $\sC$ is a $k$-linear algebraic
triangulated category which is idempotent complete and classically
generated by a $(d+1)$-spherical object $s$; that is,
\[
  \dim_k \sC( s , \Sigma^{\ell} s )
  = \left\{
      \begin{array}{cl}
        1 & \mbox{ for $\ell = 0, d+1$, } \\[1mm]
        0 & \mbox{ otherwise. }
      \end{array}
    \right.
\]
Note that $\sC( - , - )$ is short for the $\Hom$ functor in $\sC$.
\end{Definition}

We prove the following three theorems about $\sC$, where Theorems A
and B show {\em very simple}, respectively {\em much more complicated}
mutation behaviour.

\medskip

{\bf Theorem A. }
{\em 
Let $\sT$ be a $d$-cluster tilting subcategory of $\sC$ and let $t
\in \sT$ be indecomposable.  Then $\sT$ can be mutated at $t$ in
precisely $d$ ways.
}

{\bf Theorem B. }
{\em
Let $0 \leq \ell \leq d-1$ be given.  Then there exists a weakly
$d$-cluster tilting subcategory $\sT_{\ell}$ of $\sC$ with an
indecomposable object $t$ such that $\sT_{\ell}$ can be mutated at $t$
in precisely $\ell$ ways.
}

{\bf Theorem C. }
{\em 
Let $\sT$ be a weakly $d$-cluster tilting subcategory of $\sC$ and
let $t \in \sT$ be indecomposable.  Then $\sT$ can be mutated at $t$
in at most $d$ ways.
}

\medskip

The interest of Theorems A and C depends on a rich supply of (weakly)
$d$-cluster tilting subcategories in $\sC$.  Indeed, such a supply
exists by the following two theorems.  As a prelude, note that there
is a bijection between subcategories $\sT \subseteq \sC$ closed under
direct sums and summands, and sets of $d$-admissible arcs $\fT$; see
Section \ref{sec:arcs}, in particular Proposition
\ref{pro:subcategory_bijection}.  A $d$-admissible arc is an arc in
the upper half plane connecting two integers $t$, $u$ with $u - t \geq
2$ and $u - t \equiv 1 \pmod d$.

\medskip

{\bf Theorem D. }
{\em 
The subcategory $\sT$ is weakly $d$-cluster tilting if and only if
the corresponding set of $d$-admissible arcs $\fT$ is a $( d+2
)$-angulation of the $\infty$-gon. 
}

{\bf Theorem E. }
{\em
The subcategory $\sT$ is $d$-cluster tilting if and only if the
corresponding set of $d$-admissible arcs $\fT$ is a $( d+2 )$-angulation
of the $\infty$-gon which is either locally finite or has a fountain. 
}

\medskip

We defer the definition of ``$( d+2 )$-angulation of the
$\infty$-gon'' and other unexplained notions to Definition
\ref{def:arcs} and merely offer Figure \ref{fig:4-angulation} which
shows part of a $4$-angulation of the $\infty$-gon with a fountain at
$0$.
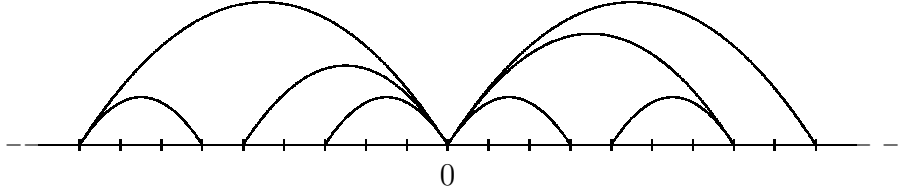
\begin{figure}
\[
 \xymatrix @-5.6pc @R=-9ex @! {
      \rule{0ex}{10.0ex} \ar@{--}[r]
    & *{}\ar@{-}[r]
    & *{\rule{0.1ex}{0.8ex}} \ar@{-}[r] 
    & *{\rule{0.1ex}{0.8ex}} \ar@{-}[r] 
    & *{\rule{0.1ex}{0.8ex}} \ar@{-}[r] 
    & *{\rule{0.1ex}{0.8ex}} \ar@{-}[r] \ar@/^-1.5pc/@{-}[lll]
    & *{\rule{0.1ex}{0.8ex}} \ar@{-}[r] 
    & *{\rule{0.1ex}{0.8ex}} \ar@{-}[r] 
    & *{\rule{0.1ex}{0.8ex}} \ar@{-}[r] 
    & *{\rule{0.1ex}{0.8ex}} \ar@{-}[r] 
    & *{\rule{0.1ex}{0.8ex}} \ar@{-}[r] 
    & *{\rule{0.1ex}{0.8ex}} \ar@{-}[r] \ar@/^-1.5pc/@{-}[lll] \ar@/^-2.5pc/@{-}[lllll]\ar@/^-4.5pc/@{-}[lllllllll]\ar@/^1.5pc/@{-}[rrr]\ar@/^3.5pc/@{-}[rrrrrrr]\ar@/^4.5pc/@{-}[rrrrrrrrr]
    & *{\rule{0.1ex}{0.8ex}} \ar@{-}[r] 
    & *{\rule{0.1ex}{0.8ex}} \ar@{-}[r]   
    & *{\rule{0.1ex}{0.8ex}} \ar@{-}[r]
    & *{\rule{0.1ex}{0.8ex}} \ar@{-}[r] 
    & *{\rule{0.1ex}{0.8ex}} \ar@{-}[r] 
    & *{\rule{0.1ex}{0.8ex}} \ar@{-}[r] 
    & *{\rule{0.1ex}{0.8ex}} \ar@{-}[r] \ar@/^-1.5pc/@{-}[lll] 
    & *{\rule{0.1ex}{0.8ex}} \ar@{-}[r] 
    & *{\rule{0.1ex}{0.8ex}} \ar@{-}[r]   
    & *{}\ar@{--}[r]
    & *{} \\
    & & & & & & & & & & & 0 \\
                   }
\]
\caption{Part of a $4$-angulation of the $\infty$-gon.}
\label{fig:4-angulation}
\end{figure}
Note how the arcs divide the upper half plane into a collection of
`quadrangular' regions, each with four integers as `vertices'.  Some
of the vertices sit at cusps.

We end the introduction with a few remarks about the category $\sC$
which has been studied intensively in a number of recent papers
\cite{FY}, \cite{HJ}, \cite{HJY}, \cite{J}, \cite{KYZ}, \cite{Ng},
\cite{ZZ}.  It is determined up to triangulated equivalence by
\cite[thm.\ 2.1]{KYZ}.  It is a Krull-Schmidt and $( d+1 )$-Calabi-Yau
category by \cite[rmk.\ 1 and prop.\ 1.8]{HJY}, and a number of other
properties can be found in \cite[secs.\ 1 and 2]{HJY}.  Theorems A and
E are two reasons for viewing $\sC$ as a cluster category of type
$A_{\infty}$, since they are infinite versions of the corresponding
theorems in type $A_n$; see \cite[thm.\ 3]{T} for Theorem A and
\cite[prop.\ 2.13]{M} and \cite[thm.\ 1]{T} for Theorem E.
See also \cite{HJ} for the case $d = 1$.

The paper is organised as follows: Section \ref{sec:arcs} introduces
$d$-admissible arcs into the study of the triangulated category $\sC$
and proves Theorem D.  Section \ref{sec:approximations} proves Theorem
E.  Section \ref{sec:combinatorics} shows some technical results on $(
d+2 )$-angulations of the $\infty$-gon.  Section \ref{sec:proofs}
proves Theorems A, B, and C.

\begin{Notation}
\label{not:lax}
We write $\ind( \sC )$ for the set of isomorphism classes of
indecomposable objects in $\sC$.  We will follow the custom of being
lax about the distinction between {\em indecomposable objects} and
{\em isomorphism classes of indecomposable objects}.  This makes the
language a bit less precise, but avoids excessive elaborations.

The word {\em subcategory} will always mean {\em full subcategory
  closed under isomorphisms, direct sums, and direct summands}.  In
particular, a subcategory is determined by the indecomposable objects
it contains.
\end{Notation}

\section{The arc picture of $\sC$}
\label{sec:arcs}

\begin{Remark}
By \cite[prop.\ 1.10]{HJY}, the Auslander-Reiten (AR) quiver of $\sC$
consists of $d$ components, each of which is a copy of $\BZ
A_{\infty}$, and $\Sigma$ acts cyclically on the set of components.
\end{Remark}

\begin{Construction}
\label{con:coordinates}
We pick a component of the AR quiver of $\sC$ and impose the
coordinate system in Figure \ref{fig:coordinate_system}.
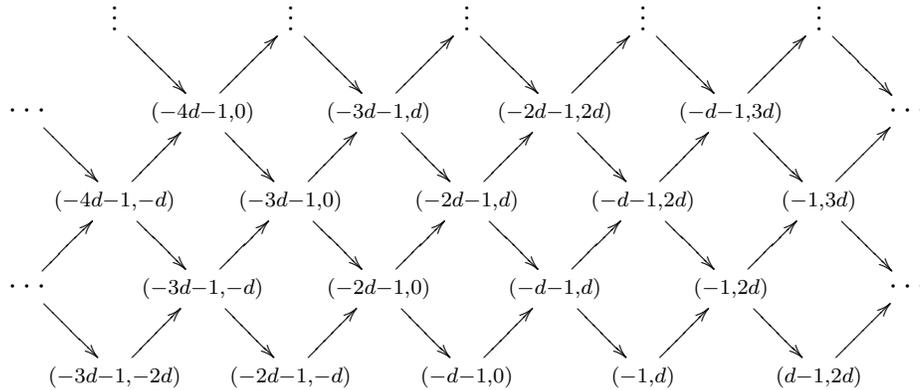
\begin{figure}
\[
  \xymatrix @-4.0pc @! {
    & \vdots \ar[dr] & & \vdots \ar[dr] & & \vdots \ar[dr] & & \vdots \ar[dr] & & \vdots \ar[dr] & \\
    \cdots \ar[dr]& & {\scriptstyle (-4d-1,0)} \ar[ur] \ar[dr] & & {\scriptstyle (-3d-1,d)} \ar[ur] \ar[dr] & & {\scriptstyle (-2d-1,2d)} \ar[ur] \ar[dr] & & {\scriptstyle (-d-1,3d)} \ar[ur] \ar[dr] & & \cdots \\
    & {\scriptstyle (-4d-1,-d)} \ar[ur] \ar[dr] & & {\scriptstyle (-3d-1,0)} \ar[ur] \ar[dr] & & {\scriptstyle (-2d-1,d)} \ar[ur] \ar[dr] & & {\scriptstyle (-d-1,2d)} \ar[ur] \ar[dr] & & {\scriptstyle (-1,3d)} \ar[ur] \ar[dr] & \\
    \cdots \ar[ur]\ar[dr]& & {\scriptstyle (-3d-1,-d)} \ar[ur] \ar[dr] & & {\scriptstyle (-2d-1,0)} \ar[ur] \ar[dr] & & {\scriptstyle (-d-1,d)} \ar[ur] \ar[dr] & & {\scriptstyle (-1,2d)} \ar[ur] \ar[dr] & & \cdots\\
    & {\scriptstyle (-3d-1,-2d)} \ar[ur] & & {\scriptstyle (-2d-1,-d)} \ar[ur] & & {\scriptstyle (-d-1,0)} \ar[ur] & & {\scriptstyle (-1,d)} \ar[ur] & & {\scriptstyle (d-1,2d)} \ar[ur] & \\
               }
\]
\caption{The coordinate system on one of the components of the AR
quiver of $\sC$.}
\label{fig:coordinate_system}
\end{figure}
We think of coordinate pairs as indecomposable objects of $\sC$, and
extend the coordinate system to the other components of the quiver by
setting
\begin{equation}
\label{equ:Sigma}
  \Sigma( t , u ) = ( t - 1 , u - 1 ).
\end{equation}
By \cite[prop.\ 1.8]{HJY}, the Serre functor of $\sC$ is $S =
\Sigma^{d+1}$.  The actions of $S$ and the AR
translation $\tau = S\Sigma^{-1}$ are given on objects by
\begin{equation}
\label{equ:S_tau}
  S( t , u ) = ( t - d - 1 , u - d - 1 ),
  \;\;\;
  \tau( t , u ) = ( t - d , u - d ).
\end{equation}
Like $\Sigma$, the Serre functor $S$ acts cyclically on the set of
components of the AR quiver.  Indeed, since there are $d$ components,
the two functors have the same action on the set of components.  The
AR translation $\tau$ is given on each component of the AR quiver by
moving one vertex to the left.
\end{Construction}

We also think of the coordinate pair $( t,u )$ as an arc in the upper
half plane connecting the integers $t$ and $u$.  The ensuing
geometrical picture is illustrated by Figure \ref{fig:4-angulation}.
However, not all values of $( t , u )$ are possible.  Indeed, it is
easy to check that the coordinate pairs which occur in Construction
\ref{con:coordinates} are precisely the $d$-admissible arcs in the
following definition.

\begin{Definition}
\label{def:arcs}
A pair of integers $( t , u )$ with $u - t \geq 2$ and $u - t \equiv 1
\pmod d$ is called a {\em $d$-admissible arc}.

The {\em length} of the arc $( t , u )$ is $u - t$.

The arcs $( r , s )$ and $( t , u )$ {\em cross} if $r < t < s < u$ or
$t < r < u < s$.  Moreover, $( r,s )$ is an {\em overarc} of $( t,u )$
if $( r,s ) \neq ( t,u )$ and $r \leq t < u \leq s$.

Let $\fT$ be a set of $d$-admissible arcs.

We say that $\fT$ is a {\em $( d+2 )$-angulation of the $\infty$-gon}
if it is a maximal set of pairwise non-crossing $d$-admissible arcs.

We say that $\fT$ is {\em locally finite} if, for each integer $t$,
there are only finitely many arcs of the form $( s , t )$ and $( t , u
)$ in $\fT$.

An integer $t$ is a {\em left-fountain of $\fT$} if $\fT$ contains
infinitely many arcs of the form $( s , t )$, and $t$ is a {\em
right-fountain of $\fT$} if $\fT$ contains infinitely many arcs of the
form $( t , u )$.  We say that $t$ is a {\em fountain} of $\fT$ if it
is both a left- and a right-fountain of $\fT$.
\end{Definition}

The first part of the following proposition is a consequence of what
we did above.  The second part follows from the first because our
subcategories are determined by the indecomposable objects they
contain, see Notation \ref{not:lax}.

\begin{Proposition}
\label{pro:subcategory_bijection}
Construction \ref{con:coordinates} gives a bijective correspondence
between $\ind( \sC )$ and the set of $d$-admissible arcs.

This extends to a bijective correspondence between {\rm (i)}
subcategories of $\sC$ and {\rm (ii)} subsets of the set of
$d$-admissible arcs.
\end{Proposition}

\begin{Definition}
Let $x \in \ind( \sC )$ be given.  Figure \ref{fig:5} defines two
infinite sets $F^{\pm}( x )$ consisting of vertices in the same
component of the AR quiver as $x$.  Each set contains $x$ and all
other vertices inside the indicated boundaries; the boundaries are
included in the sets.
\begin{figure}
\[
\vcenter{
  \xymatrix @-3.0pc @! {
    &&&*{} &&&&&&&& \\
    &&&& *{} \ar@{--}[ul] & & & & *{} \ar@{--}[ur] \\
    &*{}&& F^-(x) & & & & & & F^+(x) && *{}\\
    &&*{}\ar@{--}[ul]& & & & {x} \ar@{-}[ddll] \ar@{-}[uull] \ar@{-}[ddrr] \ar@{-}[uurr]& & &&*{}\ar@{--}[ur]&\\ 
    && \\
    *{}\ar@{--}[r]&*{} \ar@{-}[rrr] && & *{}\ar@{-}[uull]\ar@{-}[rrrrrr]& & & & \ar@{-}[uurr]\ar@{-}[rrr]&&&*{}\ar@{--}[r]&*{}\\
           }
}
\]
\caption{The sets $F^{\pm}(x)$.}
\label{fig:5}
\end{figure}
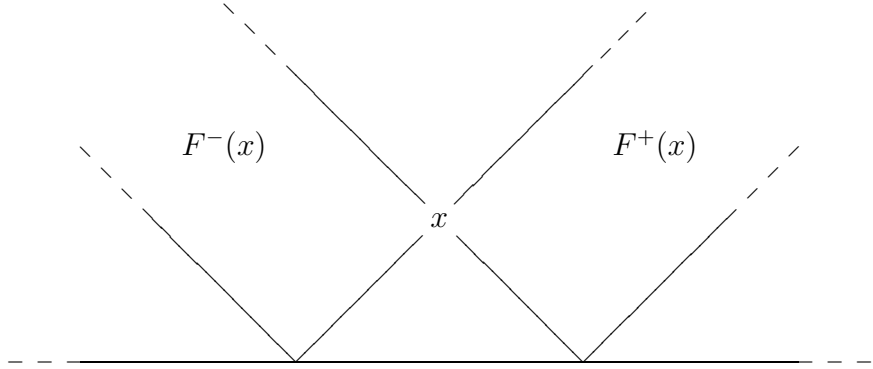
\end{Definition}

Recall that $S = \Sigma^{d+1}$ is the Serre functor of $\sC$.

\begin{Proposition}
\label{pro:Hom}
Let $x, y \in \ind( \sC )$.  Then
\[
  \dim_k \sC( x , y ) = 
  \left\{
    \begin{array}{cl}
      1 & \mbox{for $y \in F^+( x ) \cup F^-( Sx )$}, \\[2pt]
      0 & \mbox{otherwise}.
    \end{array}
  \right.
\]
\end{Proposition}

\begin{proof}
See \cite[prop.\ 2.2]{HJY}.
\end{proof}

In other words, $x$ has non-zero maps to a region $F^+( x )$ in the
same component of the AR quiver as itself, and to a region $F^-( Sx )$
in the ``next'' component of the AR quiver.  Note that if $d = 1$ then
the quiver has only one component so $F^-( Sx )$ is in the same
component as $x$.

\begin{Remark}
\label{rmk:Hom}
It is not hard to check that $y \in F^+( x ) \Leftrightarrow x \in
F^-( y )$.  So the proposition is equivalent to
\[
  \dim_k \sC( x , y ) = 
  \left\{
    \begin{array}{cl}
      1 & \mbox{for $x \in F^+( S^{ -1 } y ) \cup  F^-( y )$}, \\[2pt]
      0 & \mbox{otherwise}.

    \end{array}
  \right.
\]
\end{Remark}

The following proposition is simple but crucial since it leads
straight to Theorem D.

\begin{Proposition}
\label{pro:crossing}
Let $\fx, \fy$ be $d$-admissible arcs corresponding to $x, y \in
\ind( \sC )$.  Then $\fx$ and $\fy$ cross if and only if at least one
of the $\Hom$-spaces
\[
  \sC( x , \Sigma^1 y )
  \;\; , \;\; \ldots \;\; , \;\;
  \sC( x , \Sigma^d y )
\]
is non-zero.
\end{Proposition}

\begin{proof}
For $1 \leq \ell \leq d$, the condition that
$\sC( x , \Sigma^{ \ell }y ) \neq 0$ is equivalent to $\Sigma^{ \ell
}y \in F^+( x )$ or $\Sigma^{ \ell }y \in F^-( Sx )$ by Proposition
\ref{pro:Hom}.  If we write $\fx = ( r,s )$, $\fy = ( t,u )$, then,
using equations \eqref{equ:Sigma} and \eqref{equ:S_tau} and the
coordinate system on the AR quiver of $\sC$, it is elementary to check
that
\begin{align}
\label{equ:b}
  \Sigma^{ \ell }y \in F^+( x )
  & \Leftrightarrow
    \left\{
      \begin{array}{l}
        u \equiv s + \ell \!\!\!\!\! \pmod d, \\[1mm]
        r + \ell \leq t \leq s + \ell - d - 1, \\[1mm]
        s + \ell \leq u,
      \end{array}
    \right. \\
\label{equ:c}
  \Sigma^{ \ell }y \in F^- ( Sx )
  & \Leftrightarrow
    \left\{
      \begin{array}{l}
        u \equiv s + \ell - 1 \!\!\!\!\! \pmod d, \\[1mm]
        t \leq r + \ell - d - 1, \\[1mm]
        r + \ell \leq u \leq s + \ell - d - 1.
      \end{array}
    \right.
\end{align}
The condition that at least one of the $\Hom$ spaces $\sC( x ,
\Sigma^1 y ), \ldots, \sC( x , \Sigma^d y )$ is non-zero is hence
equivalent to the existence of at least one $\ell$ with $1 \leq \ell
\leq d$ such that the right hand side of \eqref{equ:b} or
\eqref{equ:c} is true.  It is again elementary to check that this is
equivalent to the condition that $\fx = ( r,s )$ and $\fy = ( t,u )$
cross. 
\end{proof}

{\bf Proof of Theorem D. }
Combine the definition of weakly $d$-cluster tilting subcategories
with Propositions \ref{pro:subcategory_bijection} and
\ref{pro:crossing}.
\hfill $\Box$

\section{Left- and right-approximating subcategories}
\label{sec:approximations}

\begin{Proposition}
\label{pro:forward_factorization}
Let $x, y \in \ind( \sC )$ be such that $y \in F^+( x )$.
\begin{enumerate}

  \item  Each morphism $x \rightarrow y$ is a scalar multiple of a
composition of irreducible morphisms.

\smallskip

  \item  A morphism $x \rightarrow y$ which is a composition of
irreducible morphisms is non-zero.
\end{enumerate}
Keeping $x$, $y$ as above, let $z \in \ind( \sC )$ be such that $z
\in F^+( x ) \cap F^+( y )$. 
\begin{enumerate}
\setcounter{enumi}{2}

  \item  Non-zero morphisms $x \rightarrow y$, $y \rightarrow z$
  compose to a non-zero morphism $x \rightarrow z$.

\smallskip

  \item  If $y \stackrel{\psi}{\rightarrow} z$ is a non-zero morphism,
  then each morphism $x \rightarrow z$ factors as $x \rightarrow y
  \stackrel{\psi}{\rightarrow} z$.

\end{enumerate}
\end{Proposition}

\begin{proof}
(i) Let $x \stackrel{\varphi}{\rightarrow} y$ be a morphism.  If $y
\cong x$ then it follows from Proposition \ref{pro:Hom} that $\varphi$
is a scalar multiple of the identity, and then we can take the claimed
composition of irreducible morphisms to be empty.

If $y \not\cong x$, then let $\tau y \rightarrow y_1
\stackrel{\theta}{\rightarrow} y$ be the AR triangle ending in $y$.
Since $x$, $y$ are indecomposable, the morphism $\varphi$ is not a
split epimorphism so it factors as $x \rightarrow y_1
\stackrel{\theta}{\rightarrow} y$.

We can repeat this factorization process for the direct summands of
$y_1$ to which $x$ has non-zero morphisms, that is, the direct
summands of $y_1$ which are in the rectangle $R$ shown in Figure
\ref{fig:rectangle}; cf.\ Proposition \ref{pro:Hom}.
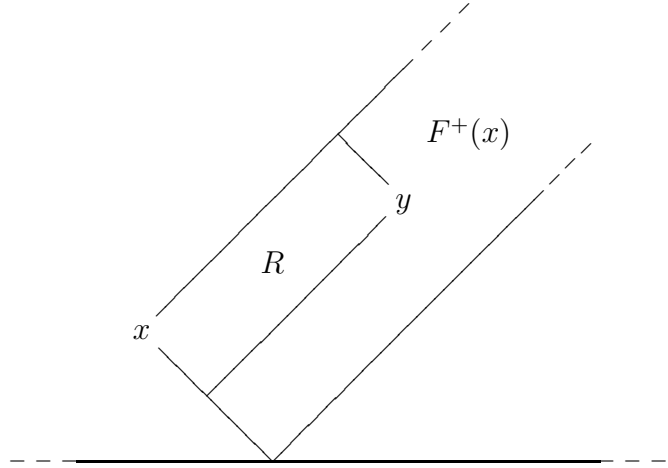
\begin{figure}
\[
  \xymatrix @-3.2pc @! {
    & & & & & & & *{} & \\
    & & & & & & *{} \ar@{--}[ur] & *{} & *{} &&\\
    & & & & & *{} & & F^+(x) & & & \\
    & & & & & & y \ar@{-}[dddlll] \ar@{-}[ul] & & *{}\ar@{--}[ur] & *{} & *{}\\
    & & & & R & & & & & \\
    *{} & & x \ar@{-}[ddrr] \ar@{-}[uuuurrrr]& & &&&&&\\ 
    & & & *{} & & & & *{} \\
    *{} \ar@{--}[r] & *{} \ar@{-}[rrr] & & & *{} \ar@{-}[uuuurrrr]\ar@{-}[rrrrr] &*{}&&&&*{}\ar@{--}[r]&*{}\\
           }
\]
\caption{The rectangle $R$ spanned by $x$ and $y$.}
\label{fig:rectangle}
\end{figure}
Successive repetitions show that the morphism $\varphi$ is a linear
combination of compositions of irreducible morphisms within $R$.
However, the mesh relations imply that any two such compositions
are scalar multiples of each other, so $\varphi$ is a composition of
irreducible morphisms. 

(ii)  By Proposition \ref{pro:Hom} there is a non-zero morphism $x
\rightarrow y$.  By part (i), it is a scalar multiple of a
composition of irreducible morphisms.  But as remarked in the proof of
part (i), two morphisms $x \rightarrow y$ which are both compositions
of irreducible morphisms are scalar multiples of each other, so it
follows that any such composition is non-zero.

(iii)  By part (i), each of the morphisms $x \rightarrow y$ and $y
\rightarrow z$ is a scalar multiple of a composition of irreducible
morphisms, so the same is true for the composition $x \rightarrow z$.
But $z$ is in $F^+(x)$, so $x \rightarrow z$ is non-zero by part
(ii). 

(iv) By Proposition \ref{pro:Hom} there is a non-zero morphism $x
\stackrel{\varphi}{\rightarrow} y$.  The composition $x
\stackrel{\psi\varphi}{\rightarrow} z$ is non-zero by part (iii).  But
the space $\sC( x , z )$ is $1$-dimensional by Proposition
\ref{pro:Hom}, so any morphism $x \rightarrow z$ can be factored as
$\psi \circ \alpha \varphi$ with $\alpha$ a scalar.
\end{proof}

\begin{Proposition}
\label{pro:backward_factorization}
Let $x, y, z \in \ind( \sC )$ be such that
$x \in F^+( S^{ -1 }y ) \cap F^+( S^{ -1 }z )$
and $z \in F^+(y)$.

If $y \stackrel{\psi}{\rightarrow} z$ is a non-zero morphism,
then each morphism $x \rightarrow z$ factors as $x \rightarrow y
\stackrel{\psi}{\rightarrow} z$.
\end{Proposition}

\begin{proof}
We must show that $\sC( x,\psi ) : \sC( x,y ) \rightarrow \sC( x,z )$
is surjective.  By Serre duality, it is equivalent to show that $\sC(
\psi,Sx ) : \sC( z,Sx ) \rightarrow \sC( y,Sx )$ is injective.  For
this it is enough to show that $\sC( \psi,Sx )$ is non-zero, since the
$\Hom$ spaces $\sC( z,Sx )$ and $\sC( y,Sx )$ have dimension $0$ or
$1$ over the ground field $k$ by Proposition \ref{pro:Hom}.

We must hence show that if $z \rightarrow Sx$ is non-zero, then so is
the composition $y \stackrel{\psi}{\rightarrow} z \rightarrow Sx$.
And this holds by Proposition \ref{pro:forward_factorization}(iii)
since we have $z \in F^+(y)$ and $Sx \in F^+( y ) \cap F^+( z )$; the
latter condition holds because it is equivalent to the assumption $x
\in F^+( S^{ -1 }y ) \cap F^+( S^{ -1 }z )$.
\end{proof}

\begin{Remark}
\label{rmk:approximating}
Recall that if $\sS$ is a subcategory of $\sC$ and $x \in \sC$ is
an object, then a right-$\sS$-approximation of $x$ is a morphism $s
\stackrel{\sigma}{\rightarrow} x$ with $s \in \sS$ such that each
morphism $s' \rightarrow x$ with $s' \in \sS$ factors through
$\sigma$.

If each $x \in \sC$ has a right-$\sS$-approximation, then $\sS$ is
called right-approximating.  There are dual notions with ``left''
instead of ``right''.
\end{Remark}

The following is a generalization of \cite[thm.\ 4.4]{HJ} and
\cite[thm.\ 2.2]{Ng}, and we follow the proofs of those results.

\begin{Proposition}
\label{pro:basic_d_neq_1}
Let $\sS$ be a subcategory of $\sC$ and let $\fS$ be the corresponding
set of $d$-admissible arcs.  The following conditions are equivalent.
\begin{enumerate}

  \item  The subcategory $\sS$ is right-approximating.

\smallskip

  \item  Each right-fountain of $\fS$ is a left-fountain of $\fS$.

\end{enumerate}
\end{Proposition}

\begin{proof}
For $d = 1$ this is \cite[thm.\ 2.2]{Ng} so assume $d \geq 2$.

Recall the notion of a slice: If $( t , u )$ is a vertex on the base
line of the AR quiver of $\sC$, then the slice starting at $( t , u )$
is $( t , * )$; that is, it consists of the vertices with coordinates
of the form $( t , u' )$.  The slice ending at $( t , u )$ is $( * , u
)$.

This means that $t \in \BZ$ is a right-fountain of $\fS$ if and only
if $\sS$ has infinitely many indecomposable objects on the slice $( t
, * )$ starting at $( t , t + d + 1 )$.
Likewise, $t$ is a left-fountain of $\fS$ if and only if $\sS$ has
infinitely many indecomposable objects on the slice $( * , t )$ ending
at $( t - d - 1 , t ) = S( t , t + d + 1 )$.  Hence (ii) is equivalent
to the following condition on $\sS$.
\begin{itemize}

  \item[(ii')]  Let $v \in \ind( \sC )$ be on the base line of the AR
  quiver of $\sC$.  If $\sS$ has infinitely many indecomposable
  objects on the slice starting at $v$, then it has infinitely many
  indecomposable objects on the slice ending at $Sv$.

\end{itemize}
Strictly speaking, we should say ``infinitely many isomorphism classes
of indecomposable objects'' but as mentioned in Notation
\ref{not:lax} we are lax about this.

(i) $\Rightarrow$ (ii').
Let $v \in \ind( \sC )$ be on the base line of the AR quiver.  Note
that $v$ and $Sv$ are in different components of the AR quiver since
there are $d \geq 2$ components and $S$ moves vertices to the ``next''
component; cf.\ Construction \ref{con:coordinates}.
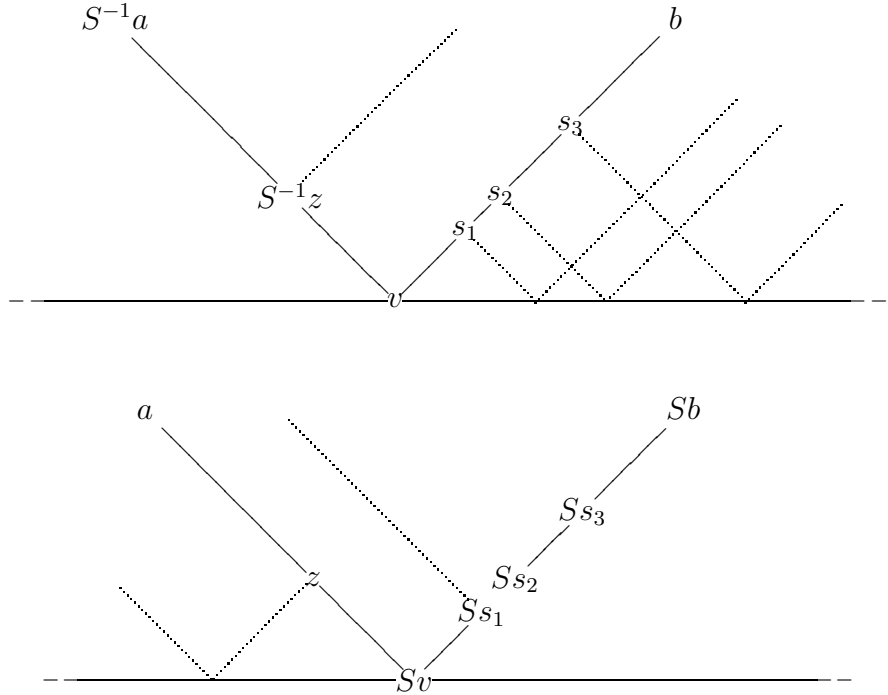
\begin{figure}
\begin{eqnarray*}
&
\vcenter{
  \xymatrix @-3.6pc @! {
    &&&{S^{-1}a}&&&&&&&*{}&&&&&&&&&{b} \\
    &&&&&&&&*{}&&&&&&&&*{}&&*{}&&&&*{}\\
    &&&&&&&*{}&&&&&&&&&&*{}&&&&&*{}&*{}\\
    &&&&&&&&*{} &&&&&*{}& *{} & *{} & *{s_3} \ar@{-}[uuurrr]\ar@{.}[dddddrrrrr] &&&&*{}&&*{} \\
    &&&&&&&& & *{} & & & & & & *{} &&&&*{}&*{}&*{}&&&&*{}\\
    &&&&&&&&*{S^{-1}z} \ar@{-}[uuuuulllll]\ar@{.}[uuuuurrrrr]& & & & *{}& & *{s_2} \ar@{-}[uurr] \ar@{.}[dddrrr]& & && *{}&&&&&&\\
    &&&&&&&& & & & *{} & & *{s_1} \ar@{.}[ddrr] \ar@{-}[ur]& & & &*{}&&&&&&*{}\\ 
    &&&&&&& & & & & & & &&&& \\
    *{}\ar@{--}[r]&*{} \ar@{-}[rrrrrrrr] &&&&&&& & *{}\ar@{-}[rr]& *{} & *{v}  \ar@{-}[uuulll]\ar@{-}[uurr]\ar@{-}[rrrr]& & & & *{} \ar@{.}[uuuuuurrrrrr]\ar@{-}[rrrrrrrrr]&&\ar@{.}[uuuuurrrrr]&*{}&*{}&&*{}\ar@{.}[uuurrr]&&&*{}\ar@{--}[r]&*{}\\
           }
}
& \\
& & \\
& & \\
&
\vcenter{
  \xymatrix @-2.6pc @! {
    &&&{a}&&&&&&&*{}&&&&&&&&&{Sb} \\
    &&&&&&&&*{}&&&&&&&&*{}&&*{}&&&&*{}\\
    &&&&&&&*{}&&&&&&&&&&*{}&&&&&*{}&*{}\\
    &&&&&&&&*{} &&&&&*{}& *{} & *{} & *{Ss_3} \ar@{-}[uuurrr] &&&&*{}&&*{} \\
    &&&&&&&& & *{} & & & & & & *{} &&&&*{}&*{}&*{}&&&*{}\\
    &&&&&&&&*{z} \ar@{.}[dddlll]\ar@{-}[uuuuulllll]& & & & *{}& & *{Ss_2} \ar@{-}[uurr] & & && *{}\\
    &&&&&&&& & & & *{} & & *{Ss_1} \ar@{.}[uuuuuullllll] \ar@{-}[ur]& & & &*{}&&&&&&*{}\\ 
    &&&&&&& & & & & & & &&&& \\
    *{}\ar@{--}[r]&*{} \ar@{-}[rrrrrrrr] &&&&\ar@{.}[uuulll]&&& & *{}\ar@{-}[rr]& *{} & *{Sv}  \ar@{-}[uuulll]\ar@{-}[uurr]\ar@{-}[rrrr]& & & & *{} \ar@{-}[rrrrrrrr]&&&&*{}&*{}&&*{}&*{}\ar@{--}[r]&*{}\\
           }
}
&
\end{eqnarray*}
\caption{Objects in two components of the AR quiver.}
\label{fig:1}
\end{figure}
Figure \ref{fig:1} shows the components of the quiver containing $v$
and $Sv$.  As indicated, $b$ is the slice starting at $v$ and $a$ the
slice ending at $Sv$.  Assume that (i) holds and that $\ind( \sS )
\cap b$ is infinite.  To show (ii'), we must show that $\ind( \sS )
\cap a$ is infinite.

Let $z$ be an indecomposable object on $a$ with
right-$\sS$-approximation $s \stackrel{\sigma}{\rightarrow} z$.  If
$s_1$ is an object on $b$ then as shown by outlines in the figure we
have $z \in F^-(Ss_1)$.  Hence there is a non-zero morphism $s_1
\rightarrow z$ by Proposition \ref{pro:Hom}.  So each of the
infinitely many objects in $\ind( \sS ) \cap b$ has a non-zero
morphism to $z$, and each such morphism factors through $\sigma$
because $\sigma$ is a right-$\sS$-approximation.  Since $\sC$ is a
Krull-Schmidt category, this implies that there is an indecomposable
direct summand $s'$ of $s$ such that the component $s'
\stackrel{\sigma'}{\rightarrow} z$ of $\sigma$ is non-zero and such
that there are infinitely many objects $s_1$, $s_2$, $s_3$, $\ldots$
in $\ind( \sS ) \cap b$ which have non-zero morphisms to $s'$.  Note
that $s' \in \sS$ since $\sS$ is closed under direct summands by
assumption.

We claim that this forces $s'$ to be on $a$, higher up than $z$.
Hence, by moving $z$ upwards we obtain infinitely many objects in
$\ind( \sS ) \cap a$.

To prove the claim, note that since $s'
\stackrel{\sigma'}{\rightarrow} z$ is non-zero, Remark \ref{rmk:Hom}
gives $s' \in F^+( S^{ -1 }z ) \cup F^-( z )$.  The sets $F^+( S^{
  -1 }z )$ and $F^-( z )$ are outlined in Figure \ref{fig:1}.  And $s'
\in F^+(S^{-1}z)$ is impossible because there would not be infinitely
many objects in $\ind( \sS ) \cap b$ with a non-zero morphism to $s'$,
as one sees by considering the sets $F^+(s_i)$ which are also
outlined in the figure.

So we have $s' \in F^-(z)$.  We already know $s' \in F^-(Ss_i)$ for
each $i$.  Hence, as one sees in Figure \ref{fig:1}, we have $s'$ on
$a$.  Finally, since there is a non-zero morphism $s'
\stackrel{\sigma'}{\rightarrow} z$, it follows that $s'$ is higher up
on $a$ than $z$.

(ii') $\Rightarrow$ (i).
Assume that (ii') holds and that $z \in \ind( \sC )$ is given.  We
will show (i) by constructing a right-$\sS$-approximation $s
\stackrel{\sigma}{\rightarrow} z$.  We must ensure that each morphism
$s' \rightarrow z$ with $s' \in \sS$ factors through $\sigma$, and we
will do so by considering the possibilities for $s'$ and building up
$\sigma$ accordingly.

We only need to consider those $s' \in \ind( \sS )$ which have
non-zero morphisms to $z$.  By Remark \ref{rmk:Hom} there are the
cases $s' \in F^-( z )$ and $s' \in F^+( S^{ -1 }z )$, see Figure
\ref{fig:10}.  Note that $z$ and $S^{ -1 }z$ are in different
components of the AR quiver; cf.\ the previous part of the proof.

\begin{figure}
\begin{eqnarray*}
&
\vcenter{
  \xymatrix @-4.6pc @! {
    &&&&&&&&& \\
    &&&&&&&& *{} \ar@{--}[ur] && b \\
    && \\
    &&&&&&&&&&&& \\
    &&&&&&&&&&& *{} \ar@{--}[ur] \\
    &&&& *{S^{-1}z} \ar@{-}[uuuurrrr] \\
    &&&&&&&&&&& *{F^+(S^{-1}z)} \\ 
    && \\
    *{}\ar@{--}[r] & *{} \ar@{-}[rr] & & *{v} \ar@{.}[uuuuuuurrrrrrr] \ar@{-}[rr] && *{}\ar@{-}[rr]& *{} & *{}  \ar@{-}[uuulll]\ar@{-}[uuuurrrr]\ar@{-}[rrrr]& & & & *{} \ar@{-}[rr]&&*{}\ar@{--}[r]&*{}\\
           }
}
& \\
& & \\
& & \\
&
\vcenter{
  \xymatrix @-3.1pc @! {
    &&&&&&&&&&& \\
    &&&&&& *{} \ar@{--}[ul] &&\\
    &&& a &&&& \\
    &&&&&&& \\
    &&& \ar@{--}[ul] &&&& \\
    &&&&&&&&&&*{z} \ar@{-}[dddlll]\ar@{-}[uuuullll]\\
    &&&& *{F^-(z) \;} &&&&&&& \\ 
    &&&&&&&&&  \\
    *{}\ar@{--}[r]&*{} \ar@{-}[rrrrrrrr] &&&&&& \ar@{-}[uuuullll] && *{Sv} \ar@{.}[uuuuuullllll]\ar@{-}[rr] && *{}\ar@{--}[r]&*{}\\
           }
}
&
\end{eqnarray*}
\caption{Another view of objects in two components of the AR quiver.}
\label{fig:10}
\end{figure}
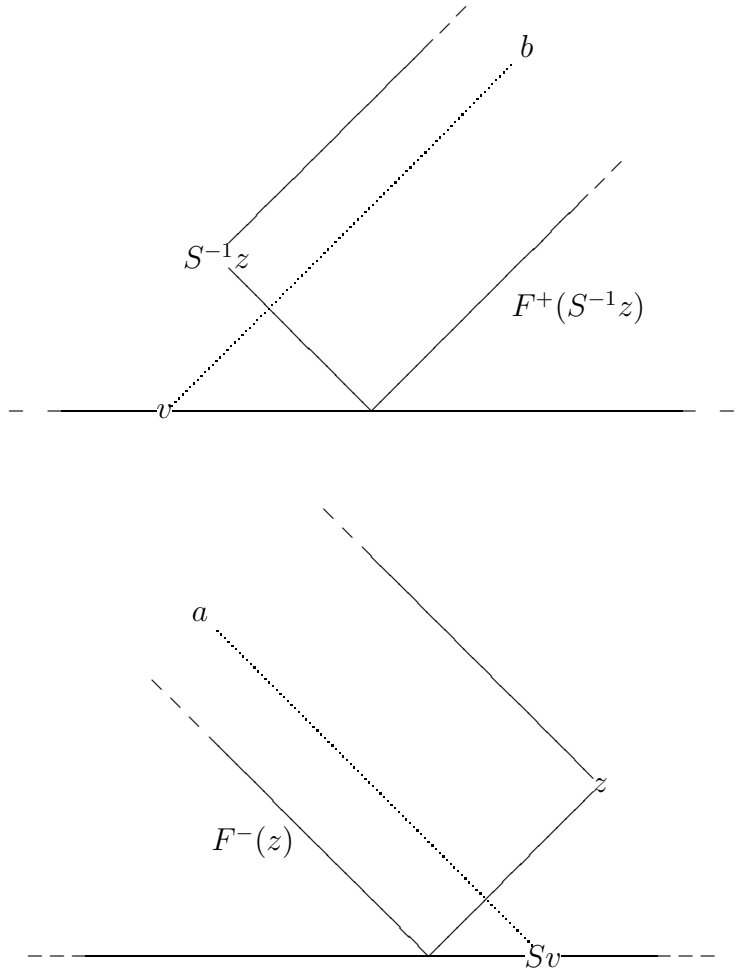

First, assume $s' \in \ind( \sS ) \cap F^-( z )$.  The slice $a$ in
Figure \ref{fig:10} determines a half line $F^-( z ) \cap a$.  If
there are objects of $\sS$ on this half line, then let $s_a$ be the
one which is closest to the base line of the quiver and let $s_a
\stackrel{ \sigma_a }{ \rightarrow } z$ be a non-zero morphism.  If
$s'$ is on $a$ then it is above $s_a$ and Proposition
\ref{pro:forward_factorization}(iv) implies that each morphism $s'
\rightarrow z$ factors through $\sigma_a$.  There are only finitely
many slices $a$ intersecting $F^-( z )$.  Including the corresponding
morphisms $\sigma_a$ as components of $\sigma$ ensures that each
morphism $s' \rightarrow z$ with $s' \in \ind( \sS ) \cap F^-( z )$
factors through $\sigma$.

Secondly, assume $s' \in \ind( \sS ) \cap F^+( S^{ -1 } z )$.  The
slice $b$ in Figure \ref{fig:10} determines a half line $F^+( S^{ -1
}z ) \cap b$, and we split into two cases.

The case where $\sS$ has finitely many objects on $F^+( S^{ -1 }z )
\cap b \;$:  Let $s_b$ be the direct sum of these objects and let each
component of the morphism $s_b \stackrel{ \sigma_b }{ \rightarrow } z$
be non-zero.  If $s'$ is on $b$, then $s'$ is one of the direct
summands of $s_b$ and each morphism $s' \rightarrow z$ factors through
$\sigma_b$ since each non-zero $\Hom$ space in $\sC$ is
$1$-dimensional.

The case where $\sS$ has infinitely many objects on $F^+( S^{ -1 }z )
\cap b \;$: Then $\ind( \sS ) \cap b$ is infinite.  If $b$ is the
slice starting at $v$ and $a$ the slice ending at $Sv$, then condition
(ii') says that $\ind( \sS ) \cap a$ is also infinite.  In particular
it is non-empty so we have already included the non-zero morphism $s_a
\stackrel{ \sigma_a }{ \rightarrow } z$ as a component of $\sigma$ in
the previous part of the proof.  If $s'$ is on $b$ then it is
straightforward to use Proposition \ref{pro:backward_factorization} to
check that each morphism $s' \rightarrow z$ factors through
$\sigma_a$.

As above, there are only finitely many slices $b$ intersecting $F^+(
S^{ -1 }z )$.  Including the relevant morphisms $\sigma_b$ as
components of $\sigma$ ensures that each morphism $s' \rightarrow z$
with $s' \in \ind( \sS ) \cap F^+( S^{ -1 }z )$ factors through
$\sigma$.
\end{proof}

A similar proof establishes the following dual result.

\begin{Proposition}
\label{pro:basic_d_neq_1_dual}
Let $\sS$ be a subcategory of $\sC$ and let $\fS$ be the corresponding
set of $d$-admissible arcs.  The following conditions are equivalent.
\begin{enumerate}

  \item  The subcategory $\sS$ is left-approximating.

\smallskip

  \item  Each left-fountain of $\fS$ is a right-fountain of $\fS$. 

\end{enumerate}
\end{Proposition}

{\bf Proof of Theorem E. }
Given a subcategory $\sT$ of $\sC$ and the corresponding set of
$d$-admissible arcs $\fT$, Theorem D says that $\sT$ is weakly
$d$-cluster tilting if and only if $\fT$ is a $( d+2 )$-angulation of
the $\infty$-gon.  It is not hard to see that since $\fT$ is a set of
pairwise non-crossing $d$-admissible arcs, it is locally finite or has
a fountain if and only if it satisfies conditions (ii) in Propositions
\ref{pro:basic_d_neq_1} and \ref{pro:basic_d_neq_1_dual}.  By the
propositions, this happens if and only if $\sT$ is left- and
right-approximating.
\hfill $\Box$

\section{Arc combinatorics}
\label{sec:combinatorics}

\begin{Construction}
\label{con:i}
Let $\fT$ be a $( d+2 )$-angulation of the $\infty$-gon and let $p_0
\in \BZ$ be given.  We define integers $p_1, p_2, \ldots$
inductively as follows:  If $p_{ \ell }$ has already been defined,
then 
\begin{itemize}

  \item  if $\fT$ contains no arcs of the form $( p_{ \ell } , q )$,
    then let $p_{ \ell + 1 } = p_{ \ell } + 1$;

\smallskip

  \item if $\fT$ contains a non-zero, finite number of arcs of the form
    $( p_{ \ell } , q )$, then let $( p_{\ell} , p_{\ell+1} )$ be the
    one with maximal length;

\smallskip

  \item  if $\fT$ contains infinitely many arcs of the form $( p_{
      \ell } , q )$, that is, if $p_{ \ell }$ is a right-fountain of
      $\fT$, then stop the algorithm and do not define $p_{ \ell + 1 }$.  

\end{itemize}
If the algorithm stops, then it defines a sequence with finitely many
elements,
\[
  p_0 < \cdots < p_m.
\]
If it does not stop, then it defines a sequence with infinitely many
elements,
\[
  p_0 < p_1 < \cdots, 
\]
and we set $m = \infty$.  Let us sum up the properties of the
sequence.
\begin{enumerate}

  \item If $m < \infty$, then $p_m$ is a right-fountain of $\fT$.

\smallskip

  \item $( p_{\ell} , p_{ \ell+1 } )$ is either a pair of consecutive
    integers or an arc in $\fT$.

\smallskip

  \item $p_{\ell} - p_0 \equiv \ell \pmod d$.

\end{enumerate}
To see (iii), note that the length of a $d$-admissible arc is $\equiv
1 \pmod d$.
\end{Construction}

Collin Bleak proved that a triangulation of the $\infty$-gon has a
left-fountain if and only if it has a right-fountain, and his method
also works for $( d+2 )$-angulations.  We thank him for permitting us
to provide a proof of the following lemma which establishes the ``only
if'' direction.  ``If'' follows by symmetry.  See also \cite[lem.\
4.11]{G}.

\begin{Lemma}
\label{lem:fountains2}
Let $\fT$ be a $( d+2 )$-angulation of the $\infty$-gon.  Suppose that
$p_0$ is a left-fountain of $\fT$ and perform Construction
\ref{con:i}. 
\begin{enumerate}

  \item The construction gives a finite sequence $p_0 < \cdots < p_m$
    with $0 \leq m \leq d$.

\smallskip

  \item  $p_m$ is a right-fountain of $\fT$.

\smallskip

  \item  Let $\ft \in \fT$ and assume $\ft \neq ( p_{\ell} , p_{\ell +
  1 })$ for $\ell \in \{ 0 , \ldots, m-1 \}$.  Then $\ft$ has an
  overarc $\fr \in \fT$.

\end{enumerate}
\end{Lemma}

\begin{proof}
(i) Assume to the contrary that $m \geq d + 1$; this includes the
possibility $m = \infty$.  Construction \ref{con:i}(iii) shows that $(
p_0 , p_{ d+1 } )$ is a $d$-admissible arc.  If $( p_0 , p_1 )$ is a
$d$-admissible arc, then $( p_0 , p_{ d+1 } )$ has strictly greater
length than $( p_0 , p_1)$ and by Construction \ref{con:i} we have $(
p_0 , p_{ d+1 } ) \not\in \fT$.  If $( p_0 , p_1 )$ is not a
$d$-admissible arc, then by Construction \ref{con:i} there are no arcs
in $\fT$ of the form $( p_0 , q )$ so we have $( p_0 , p_{
  d+1 } ) \not\in \fT$ again.  In either case there must be an arc $(
r,s ) \in \fT$ which crosses $( p_0 , p_{ d+1 } )$, that is, $r < p_0
< s < p_{ d+1 }$ or $p_0 < r < p_{ d+1 } < s$.  But $p_0$ is a
left-fountain of $\fT$ so $r < p_0 < s < p_{ d+1 }$ is impossible
since it would imply that $( r,s )$ crossed an arc in $\fT$.

We must therefore have $p_0 < r < p_{ d+1 } < s$.  However, this also
leads to a contradiction: We cannot have $p_{\ell} < r < p_{\ell+1}$
for any $\ell \in \{ 0, \ldots, d \}$, for if we did then $( p_{\ell}
, p_{ \ell+1 })$ would not be consecutive integers whence $( p_{\ell}
, p_{\ell+1} ) \in \fT$ by Construction \ref{con:i}(ii), but this arc
would cross $( r,s ) \in \fT$.  So we must have $r = p_{\ell}$ for an
$\ell \in \{ 1 , \ldots , d \}$.  Hence $\fT$ contains arcs of the
form $( p_{\ell},q )$, and by Construction \ref{con:i} the one with
maximal length is $( p_{\ell} , p_{ \ell+1 })$.  But this contradicts
$( r,s ) \in \fT$ because we know $r = p_{\ell}$ and $p_{ \ell+1 }
\leq p_{ d+1 } < s$.

(ii)  See Construction \ref{con:i}(i).

(iii)  Let us write $\ft = ( t,u )$ and search for $\fr$.

Since $p_0$ and $p_m$ are a left-fountain and a right-fountain of
$\fT$, we must have $t < u \leq p_0$ or $p_0 \leq t < u \leq p_m$ or
$p_m \leq t < u$.

If $t < u \leq p_0$, then we can choose $\fr = ( r,p_0 ) \in \fT$
with $r < t$.  If $p_m \leq t < u$, then we can choose $\fr = ( p_m,s
) \in \fT$ with $u < s$.

Now assume $p_0 \leq t < u \leq p_m$.

If $t = p_{\ell}$ for an $\ell \in \{ 0 , \ldots , m-1 \}$, then $\ft
\in \fT$ is an arc of the form $( p_{ \ell } , q )$.  Among the arcs
in $\fT$ of this form, by Construction \ref{con:i} the one with
maximal length is $( p_{\ell} , p_{\ell+1} )$.  Since $\ft \neq (
p_{\ell} , p_{\ell+1} )$ by assumption, we get that $\fr = ( p_{\ell}
, p_{\ell+1} ) \in \fT$ is an overarc of $\ft$.

If $p_{\ell} < t < p_{\ell+1}$ for an $\ell \in \{ 0 , \ldots , m-1
\}$, then we must have $u \leq p_{\ell+1}$, since otherwise $( t,u )
\in \fT$ and $( p_{\ell} , p_{\ell+1} ) \in \fT$ would cross.  But
then $\fr = ( p_{\ell} , p_{\ell+1} ) \in \fT$ is again an overarc of
$\ft = ( t,u )$.
\end{proof}

\begin{Lemma}
\label{lem:fountains}
Let $\fT$ be a $( d+2 )$-angulation of the $\infty$-gon.  Then $\fT$
is either locally finite or has precisely one left-fountain and one
right-fountain. 
\end{Lemma}

\begin{proof}
If $\fT$ is not locally finite, then it has a left- or a
right-fountain.  By Lemma \ref{lem:fountains2}(ii) and its mirror
image, it has both a left- and a right-fountain.  And it is easy to
see that in any event, it has at most one left- and at most one
right-fountain.
\end{proof}

\begin{Lemma}
\label{lem:overarc}
Let $\fT$ be a $( d+2 )$-angulation of the $\infty$-gon which is
locally finite or has a fountain.  Then each arc $\ft \in \fT$ has an
overarc $\fr \in \fT$.
\end{Lemma}

\begin{proof}
The case where $\fT$ has a fountain at $p_0 \;$: Then we must have $m =
0$ in Lemma \ref{lem:fountains2}, and Lemma \ref{lem:fountains2}(iii)
implies the present result.

The case where $\fT$ is locally finite: Let us write $\ft = ( t,u )$
and search for $\fr$.  We can assume that, among the arcs in $\fT$ of
the form $( t,v )$, the one of maximal length is $( t,u )$, since
otherwise there is obviously an overarc.  Let $p_0 = t$ and perform
Construction \ref{con:i}; then $ ( p_0 , p_1 ) = ( t,u )$.  Since
$\fT$ is locally finite, it has no right-fountain, so Construction
\ref{con:i}(i) implies $m = \infty$.  Construction \ref{con:i}(iii)
implies that $( p_0 , p_{ d+1 } )$ is a $d$-admissible arc.  It has
strictly greater length than $( p_0 , p_1)$, so $( p_0 , p_{ d+1 } )
\not\in \fT$ follows.

There must hence be an arc $( r,s ) \in \fT$ which crosses $( p_0 ,
p_{ d+1 } )$, that is, $r < p_0 < s < p_{ d+1 }$ or $p_0 < r < p_{ d+1
} < s$.

First, assume $r < p_0 < s < p_{ d+1 }$.  Note that we cannot have
$p_0 < s < p_1$ since then $( r,s ) \in \fT$ and $( p_0 , p_1 ) \in
\fT$ would cross.  So $p_1 \leq s$ whence $\fr = ( r,s ) \in \fT$ is
an overarc of $( p_0 , p_1 ) = ( t,u )$.

Secondly, assume $p_0 < r < p_{ d+1 } < s$.  This leads to a
contradiction in the same way as in the second paragraph of the proof
of Lemma \ref{lem:fountains2}.
\end{proof}

\begin{Construction}
\label{con:polygon}
Let $\fT$ be a $( d+2 )$-angulation of the $\infty$-gon and let $\fr =
( r,s ) \in \fT$.  We can view $\{ r , \ldots , s \}$ as the vertices
of an $( s-r+1 )$-gon $R$.  Each pair $( r , r+1 ), ( r+1,r+2 ),
\ldots, ( s-1,s )$ is viewed as an edge of $R$, and so is the arc $(
r,s )$; that is, $r$ and $s$ are viewed as consecutive vertices of
$R$.  Each $d$-admissible arc $\ft$ of which $\fr = ( r,s )$ is an
overarc is viewed as a $d$-admissible diagonal of $R$.  See Figure
\ref{fig:polygon}.
\begin{figure}
\begin{eqnarray*}
  \parbox[t]{3cm}{
    \vspace{-17ex}
$
  \xymatrix @-3.0pc @R=-4ex @! {
     & *{\rule{0ex}{5.5ex}} \ar@{-}[r]
     & *{\rule{0.1ex}{0.8ex}} \ar@{-}[r] \ar@/^1.5pc/@{-}[rrrrr]
     & *{\rule{0.1ex}{0.8ex}} \ar@{-}[r] 
     & *{\rule{0.1ex}{0.8ex}} \ar@{-}[r] \ar@/^1.0pc/@{-}[rrr]
     & *{\rule{0.1ex}{0.8ex}} \ar@{-}[r] 
     & *{\rule{0.1ex}{0.8ex}} \ar@{-}[r] 
     & *{\rule{0.1ex}{0.8ex}} \ar@{-}[r] \ar@/^1.0pc/@{-}^{\textstyle\ft}[rrr]
     & *{\rule{0.1ex}{0.8ex}} \ar@{-}[r] 
     & *{\rule{0.1ex}{0.8ex}} \ar@{-}[r] 
     & *{\rule{0.1ex}{0.8ex}} \ar@{-}[r] 
     & *{\rule{0.1ex}{0.8ex}} \ar@{-}[r] \ar@/^-3.0pc/@{-}_{\textstyle \fr}[lllllllll]
     & *{} \\
     & & r & & & & & & & & & s \\
                               }
$
                 }
& \;\;\;\;\;\;\;\;\;\;\;\;\;\;\;\;\;\;\;\;\;\;\;\;\;\;\;\;\;\;\;\;\;\;\;\;\;\; &
  \begin{tikzpicture}[auto]
    \node[name=s, shape=regular polygon, regular polygon sides=10, minimum size=4cm, draw] {}; 
    \draw[shift=(s.corner 2)] node[above] {$r$};
    \draw[shift=(s.corner 1)] node[above] {$s$};
    \draw[shift=(s.side 1)] node[above] {$\fr$};
    \draw[thick] (s.corner 2) to (s.corner 7);
    \draw[thick] (s.corner 4) to (s.corner 7);
    \draw[thick] (s.corner 7) to node[left=0pt] {$\ft$}(s.corner 10);
  \end{tikzpicture} 
\end{eqnarray*}
\caption{If $\fr = ( r,s )$ is a $d$-admissible arc then $r$, $r+1$,
  $\ldots$, $s$ can be viewed as the vertices of a polygon $R$.  If
  $\ft = ( t,u )$ has $\fr$ as an overarc, then $\ft$ can be viewed as
  a $d$-admissible diagonal of $R$.}
\label{fig:polygon}
\end{figure}
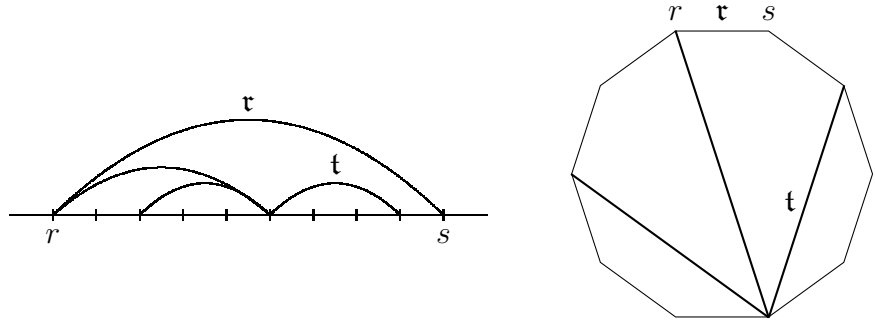
In particular, the set
\[
  \fR = \{\, \ft \in \fT \,|\, \mbox{ $\fr$ is an overarc of $\ft$ } \,\}
\]
is a $( d+2 )$-angulation of $R$.

Observe that $\fR$ divides $R$ into $( d+2 )$-gons, and that one of
these $( d+2 )$-gons, say $T$, has $r$ and $s$ among its vertices.  We
can write the whole set of vertices of $T$ as
\[
  r < t_1 < \cdots < t_d < s,
\]
and hence each of
\[
  ( r,t_1 ) \;,\; ( t_1,t_2 )
  \;,\; \ldots \;,\;
  ( t_{ d-1 },t_d ) \;,\; ( t_d,s )
\]
is either a pair of consecutive integers or a diagonal in $\fR$, that
is, an arc in $\fT$.
\end{Construction}

\begin{Lemma}
\label{lem:mutation}
Let $\fT$ be a $( d+2 )$-angulation of the $\infty$-gon and let $\ft
\in \fT$.
\begin{enumerate}

  \item If $\fU$ is a set of $d$-admissible arcs not in $\fT \setminus
  \ft$ such that $( \fT \setminus \ft ) \cup \fU$ is a $( d+2
  )$-angulation of the $\infty$-gon, then $\fU = \{ \ft^* \}$ for a
  single $d$-admissible arc $\ft^*$.

\smallskip

  \item  If $\ft$ has an overarc in $\fT$ then there are $d + 1$
  choices of $\ft^*$.

\smallskip

  \item  If $\ft$ has no overarc in $\fT$ then there are $\leq d$
  choices of $\ft^*$.

\end{enumerate}
\end{Lemma}

\begin{proof}
Suppose that $\ft$ has the overarc $\fr \in \fT$.  We will establish
(i) and (ii) for $\ft$.

The set of all arcs in $\fT$ of which $\fr \in \fT$ is an overarc can
be viewed as a $( d+2 )$-angulation $\fR$ of a polygon $R$ by
Construction \ref{con:polygon}.  When $( \fT \setminus \ft ) \cup \fU$
is a $( d+2 )$-angulation of the $\infty$-gon, $\fr$ is an overarc of
each arc in $\fU$ since removing $\ft$ does not create any room above
its overarc $\fr$.  It follows that $\fU$ can be viewed as a set of
$d$-admissible diagonals of $R$ such that $( \fR \setminus \ft)
\cup \fU$ is a $( d+2 )$-angulation of $R$.  Then it is well known
that, as desired, $\fU$ has one element which can be chosen in $d+1$
different ways.

Now suppose that $\ft$ has no overarc in $\fT$.  We will establish (i)
and (iii) for $\ft$.

Lemma \ref{lem:overarc} shows that $\fT$ is not locally finite and
does not have a fountain.  Lemma \ref{lem:fountains} shows that $\fT$
has a left-fountain $p_0$ which is not a right-fountain.  We can
perform Construction \ref{con:i}.  By Lemma \ref{lem:fountains2}(i+ii)
this gives a sequence $p_0 < \cdots < p_m$ with $m \leq d$ where $p_m$
is a right-fountain of $\fT$.  Note that $1 \leq m$ since $p_0$ is not
a right-fountain.  By Construction \ref{con:i}(ii), each $( p_{ \ell }
, p_{ \ell+1 } )$ is either a pair of consecutive integers or an arc
in $\fT$, and it follows from Lemma \ref{lem:fountains2}(iii) that
$\ft = ( p_{j} , p_{j+1} )$ for a $j \in \{ 0 , \ldots , m-1 \}$.

By Construction \ref{con:polygon} applied to $\ft = ( p_{ j } , p_{
j+1 } )$, there is a sequence of integers $p_{j} < q_1 < \cdots
< q_d < p_{ j+1 }$ such that each of $( p_{j} , q_1 )$, $( q_1 ,
q_2 )$, $\ldots$, $( q_{ d-1 } , q_d )$, $( q_d , p_{ j+1 } )$ is
either a pair of consecutive integers or an arc in $\fT$.  We hence
have a sequence of integers
\begin{equation}
\label{equ:pq}
  p_0 < p_1 < \cdots < p_{j}
  < q_1 < \cdots < q_d
  < p_{ j+1 } < \cdots < p_m
\end{equation}
where each pair of neighbouring elements is either a pair of
consecutive integers or an arc in $\fT \setminus \ft$.  In particular,
each pair of neighbouring elements has a difference which is $\equiv 1
\pmod d$.

Now consider a $d$-admissible arc $\ft^* = ( v,w ) \not\in \fT
\setminus \ft$ which crosses no arc in $\fT \setminus \ft$.

We cannot have $w \leq p_0$.  For if we did, then $\ft^* = ( v,w )$
would not cross $\ft = ( p_{j} , p_{ j+1 } )$, and hence $\ft^*$ would
cross no arc in $\fT$ whence $\ft^* \in \fT$.  Since $\ft^* \not\in
\fT \setminus \ft$, this would force $\ft^* = \ft$, but this
contradicts $v < w \leq p_0 \leq p_{ j }$.  We also cannot have $v <
p_0 < w$ because $p_0$ is a left-fountain of $\fT$.  Similarly, we
cannot have $p_m \leq v$ or $v < p_m < w$.

We conclude that $p_0 \leq v < w \leq p_m$.  We claim that, in fact,
$v$ and $w$ must be among the elements of the sequence \eqref{equ:pq}.

Namely, assume that at least one of $v$ and $w$ is not an element of
the sequence.  Then it is strictly between two such elements.  For the
sake of argument, say $q_{ \ell } < v < q_{ \ell+1 }$.  Then we cannot
have $q_{ \ell+1 } < w$, for then $\ft^* = ( v,w )$ and $( q_{ \ell }
, q_{ \ell+1 } ) \in \fT \setminus \ft$ would cross.  So we must have
$v < w \leq q_{ \ell+1 }$.  Hence $( q_{ \ell } , q_{ \ell+1 } )$
is an overarc of $\ft^* = ( v,w )$.

However, the set of all arcs in $\fT$ of which $( q_{ \ell } , q_{
  \ell+1 } ) \in \fT \setminus \ft$ is an overarc can be viewed as a
$( d+2 )$-angulation $\fR'$ of a polygon $R'$ by Construction
\ref{con:polygon}, and $\ft^*$ can be viewed as a $d$-admissible
diagonal of this polygon.  Note that $( q_{ \ell } , q_{ \ell+1 } )$
is not an overarc of $\ft = ( p_j , p_{ j+1 } )$ and so $\fR'
\subseteq \fT \setminus \ft$.  Hence the assumption that $\ft^*$
crosses none of the arcs in $\fT \setminus \ft$ means that it crosses
none of the diagonals in $\fR'$.  But then $\ft^* \in \fR'$ whence
$\ft^* \in \fT \setminus \ft$ which is a contradiction.

So we have $\ft^* = ( v,w )$ with $v$, $w$ elements in the sequence
\eqref{equ:pq}.  However, we saw that each pair of neighbouring
elements in this sequence has a difference which is $\equiv 1 \pmod
d$.  Hence, for $\ft^*$ to be a $d$-admissible arc, $v$ and $w$ must
either be neighbours in the sequence, or $nd+1$ steps apart for an
integer $n \geq 1$.  But they cannot be neighbours for then we would
have $\ft^* \in \fT \setminus \ft$, so $v$ and $w$ must be $nd+1$
steps apart in the sequence.

Since $m \leq d$ by Lemma \ref{lem:fountains2}(i), the sequence
\eqref{equ:pq} has $m + 1 + d \leq 2d + 1$ elements.  It follows that
$n = 1$ and hence two different choices of $\ft^*$ must cross each
other; this shows part (i) of the present lemma.  It also follows that
there are at most $(2d + 1) - (d+1) = d$ different choices for
$\ft^*$, showing part (iii) of the present lemma.
\end{proof}

\section{Proofs of Theorems A, B, and C}
\label{sec:proofs}

\begin{Remark}
Let $\sT$ be a weakly $d$-cluster tilting subcategory of the
triangulated category $\sC$ and let $\fT$ be the corresponding $( d+2
)$-angulation of the $\infty$-gon.

Lemma \ref{lem:mutation}(i) says that if we drop one arc from $\fT$,
then we must add precisely one other $d$-admissible arc to get a new
$( d+2 )$-angulation.

So if we drop one indecomposable object from $\sT$, then we must add
precisely one other indecomposable object to get a new weakly
$d$-cluster tilting subcategory of $\sC$.

That is, ``mutating $\sT$ at $t$'' has the expected effect of
replacing $t$ by a single other indecomposable object.
\end{Remark}

{\bf Proof of Theorem A. }
By Theorem E which was already proved in Section
\ref{sec:approximations}, a $d$-cluster tilting subcategory $\sT$ of
$\sC$ corresponds to a $( d+2 )$-angulation of the $\infty$-gon $\fT$
which is locally finite or has a fountain.

By Lemma \ref{lem:overarc}, each $\ft \in \fT$ has an overarc $\fr
\in \fT$.

By Lemma \ref{lem:mutation}(ii), this means that there are $d + 1$
different choices of a $d$-admissible arc $\ft^*$ such that $( \fT
\setminus \ft ) \cup \ft^*$ is a $( d+2 )$-angulation of the
$\infty$-gon.

Excluding the trivial choice $\ft^* = \ft$ leaves $d$ choices for
$\ft^*$ and translating back to $\sT$ shows Theorem A.
\hfill $\Box$

{\bf Proof of Theorem B. }
Let $\ell \in \{ 0, \ldots, d-1 \}$ be given.  Figure \ref{fig:bad}
shows part of a $( d+2 )$-angulation $\fT_{ \ell }$ of the
$\infty$-gon. 
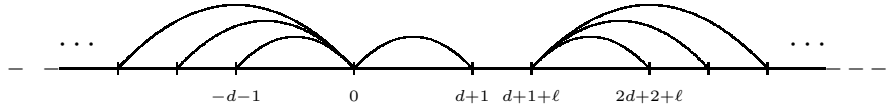
\begin{figure}
\[
\vcenter{
  \xymatrix @-3.5pc @R=-6ex @! {
       \rule{0ex}{5.5ex} \ar@{--}[r]
     & *{\rule{0ex}{4ex}^{\lefteqn{\textstyle \cdots}}} \ar@{-}[r]
     & *{\rule{0.1ex}{0.8ex}} \ar@{-}[r] \ar@/^2.0pc/@{-}[rrrr]
     & *{\rule{0.1ex}{0.8ex}} \ar@{-}[r] \ar@/^1.5pc/@{-}[rrr]
     & *{\rule{0.1ex}{0.8ex}} \ar@{-}[r] \ar@/^1.0pc/@{-}[rr]
     & *{} \ar@{-}[r] 
     & *{\rule{0.1ex}{0.8ex}} \ar@{-}[r] \ar@/^1.0pc/@{-}[rr]
     & *{} \ar@{-}[r] 
     & *{\rule{0.1ex}{0.8ex}} \ar@{-}[r] 
     & *{\rule{0.1ex}{0.8ex}} \ar@{-}[r] 
     & *{} \ar@{-}[r]
     & *{\rule{0.1ex}{0.8ex}} \ar@{-}[r] \ar@/^-1.0pc/@{-}[ll]
     & *{\rule{0.1ex}{0.8ex}} \ar@{-}[r] \ar@/^-1.5pc/@{-}[lll]
     & *{\rule{0.1ex}{0.8ex}} \ar@{-}[r] \ar@/^-2.0pc/@{-}[llll]
     & *{\rule{0ex}{4ex}^{\lefteqn{\textstyle \!\!\!\!\!\!\! \cdots}}} \ar@{--}[r]
     & *{} \\
     & & & & {\scriptscriptstyle -d-1} & & {\scriptscriptstyle 0} & & {\scriptscriptstyle d+1} & {\scriptscriptstyle d+1+\ell} & & {\scriptscriptstyle 2d+2+\ell} 
                    }
}
\]
\caption{A $( d+2 )$-angulation $\fT_{ \ell }$ of the $\infty$-gon
where the arc $\ft = ( 0,d+1 )$ can be replaced in $\ell$ ways.}
\label{fig:bad}
\end{figure}
It contains the arc $\ft = ( 0 , d+1 )$ and has a left-fountain at $0$
and a right-fountain at $d + 1 + \ell$.

There are $\ell + 1$ different choices of a $d$-admissible arc $\ft^*$
such that $( \fT \setminus \ft ) \cup \ft^*$ is a $( d+2 )$-angulation
of the $\infty$-gon; namely, $\ft^* = ( p , p + d + 1 )$ for $p \in \{
0, \ldots, \ell \}$.

Excluding the trivial choice $\ft^* = \ft$ leaves $\ell$ choices.  By
Theorem D which was already proved in Section
\ref{sec:approximations}, the $( d+2 )$-angulation $\fT_{ \ell }$
therefore corresponds to a weakly $d$-cluster tilting subcategory
$\sT_{ \ell }$ with the property claimed in Theorem B.  \hfill $\Box$

{\bf Proof of Theorem C. }
Similar to the proof of Theorem A, but $\ft \in \fT$ may or may not
have an overarc, so both parts (ii) and (iii) of Lemma
\ref{lem:mutation} are needed.
\hfill $\Box$

\medskip
\noindent
{\bf Acknowledgement.}
We are grateful
to Collin Bleak for permitting us to provide a proof that a $( d+2
)$-angulation of the $\infty$-gon has a left-fountain if and only if
it has a right-fountain, see Lemmas \ref{lem:fountains2} and
\ref{lem:fountains}.  Collin Bleak originally proved this for $d =
1$. 

This work was supported by grant number HO 1880/4-1 under the research
pri\-o\-ri\-ty programme SPP 1388 {\em
  Dar\-stel\-lungs\-the\-o\-ri\-e} of the Deut\-sche
For\-schungs\-ge\-mein\-schaft (DFG).

\end{document}